\documentclass[12pt,reqno,oneside]{amsart}

\usepackage[T1]{fontenc}
\usepackage{lmodern}
\usepackage{amsmath,amssymb,amsthm,mathtools}
\usepackage{microtype}
\usepackage[a4paper,margin=2.54cm]{geometry}
\usepackage{aliascnt}
\usepackage[colorlinks=true,linkcolor=blue,citecolor=blue,urlcolor=blue]{hyperref}
\usepackage[nameinlink,capitalize,noabbrev]{cleveref}
\hypersetup{
 pdftitle={Divergence Liftings and Regular Factorizations for Metric Cotype},
 pdfauthor={Yue Wang},
 pdfsubject={Mathematics: Functional Analysis; MSC 2020: 46B85, 46B20, 47B38, 47B65},
 pdfkeywords={metric cotype, nonlinear lifting, regular operator, quotient map, Banach lattice, L1},
 pdfdisplaydoctitle=true
}

\numberwithin{equation}{section}
\allowdisplaybreaks
\newtheorem{theorem}{Theorem}[section]
\newaliascnt{proposition}{theorem}
\newtheorem{proposition}[proposition]{Proposition}
\aliascntresetthe{proposition}
\newaliascnt{lemma}{theorem}
\newtheorem{lemma}[lemma]{Lemma}
\aliascntresetthe{lemma}
\newaliascnt{corollary}{theorem}
\newtheorem{corollary}[corollary]{Corollary}
\aliascntresetthe{corollary}
\theoremstyle{remark}
\newtheorem{remark}[theorem]{Remark}

\newcommand{\N}{\mathbb N}
\newcommand{\E}{\mathbb E}
\newcommand{\ran}{\operatorname{ran}}
\newcommand{\sgn}{\mathrm{sgn}}
\newcommand{\longop}{\mathrm{long}}
\newcommand{\norm}[1]{\left\lVert#1\right\rVert}

\title[Divergence liftings and regular factorizations]{Divergence Liftings and Regular Factorizations for Metric Cotype}
\author{Yue Wang}
\address{School of Mathematical Sciences, Xiamen University, Xiamen 361005, China}
\email{19020260158074@stu.xmu.edu.cn}
\date{}

\begin{document}
\pagestyle{plain}

\begin{abstract}
The least regular norm of a factorization of antipodal differences through sign differences on \((\mathbb Z/m\mathbb Z)^n\) is \(m/2\) when \(4\mid m\). On these tori, when \(2\sqrt n\le m\le4\sqrt n\), normalized divergence equations admit selections with pointwise norm bound \(3/2\), while the least linear lifting norm and global Lipschitz selection constant equal \(\sqrt n/2\). Higher-moment lifting bounds follow from a quoted sharp \(L_1\) inequality.
\end{abstract}

\maketitle

\section{Introduction and main results}\label{sec:introduction}

Metric-cotype inequalities on finite tori compare coordinate antipodal differences with local differences. For Banach-valued functions, these differences define bounded linear operators. Hahn--Banach duality expresses the difference inequality as solvability of the adjoint divergence equation, with the same constant. We compare the norm bound for individual solutions with the least norm of a bounded linear solution operator.

For sign differences, we compute the least regular norm of a scalar linear factorization and determine when it exists. At the \(L_\infty\) endpoint, we also obtain a separation in ordinary operator norm: individual solutions have a dimension-independent norm bound, while the least linear lifting norm and the least global Lipschitz selection constant are both \(\sqrt n/2\).

\subsection{Published results on type and cotype}

For \(d\in\{1,2,\ldots\}\), let \(\E_\varepsilon\) denote expectation over the uniform sign cube \(\{-1,1\}^d\). A Banach space \(Y\) has Rademacher type \(r\in[1,2]\) if there is a constant \(T<\infty\) such that every finite family \(y_1,\ldots,y_d\in Y\) satisfies
\[
\left(\E_\varepsilon\left\|\sum_{j=1}^d\varepsilon_jy_j\right\|_Y^2\right)^{1/2}
\le T\left(\sum_{j=1}^d\|y_j\|_Y^r\right)^{1/r}.
\]
It has Rademacher cotype \(q\in[2,\infty)\) if there is \(C<\infty\) such that
\[
\left(\sum_{j=1}^d\|y_j\|_Y^q\right)^{1/q}
\le C\left(\E_\varepsilon\left\|\sum_{j=1}^d\varepsilon_jy_j\right\|_Y^2\right)^{1/2}.
\]
The second-moment convention is equivalent to the usual moment conventions by Kahane's inequality.

The classical type--cotype duality asserts that type \(r>1\) of \(Y\) implies cotype \(r'=r/(r-1)\) of its continuous dual \(Y^*\). Under the additional assumption of \(K\)-convexity, cotype \(q\) of \(Y\) implies type \(q'=q/(q-1)\) of \(Y^*\); see \cite{Mau03}. Here \(K\)-convexity means that the Rademacher projections
\[
(P_d f)(\varepsilon)=\sum_{j=1}^d\varepsilon_j\E_\delta[\delta_j f(\delta)]
\]
on \(L_2(\{-1,1\}^d;Y)\) have operator norms bounded independently of \(d\), where \(\delta\) is uniform on the same sign cube. Finite cotype alone is insufficient: \(\ell_1\) has cotype \(2\), whereas \(\ell_\infty=\ell_1^*\) has no type \(r>1\).

Enflo type \(r\) is the metric inequality
\[
\E_\varepsilon\|f(\varepsilon)-f(-\varepsilon)\|_Y^r
\le T^r\sum_{j=1}^d\E_\varepsilon\|f(\varepsilon)-f(\varepsilon^{(j)})\|_Y^r
\]
for every \(f:\{-1,1\}^d\to Y\), where \(\varepsilon^{(j)}\) is obtained by reversing the \(j\)-th sign. Ivanisvili, van Handel, and Volberg proved that Enflo type and Rademacher type coincide for Banach spaces \cite{IVV20}. Mendel and Naor proved that metric cotype \(q\) and Rademacher cotype \(q\) coincide for Banach spaces \cite[Theorem~1.2]{MN08}. These are qualitative characterizations; for metric cotype, the order of the torus scale is an additional quantitative issue.

Sharp scaling for \(K\)-convex Banach spaces was established in \cite[Section~4]{MN08}. General scaling estimates and limitations of a smoothing-and-approximation construction were studied in \cite{GMN11}. Sign-increment estimates occur in \cite{NS16}, and comparison with ternary increments is discussed in \cite[Section~4]{EMN19}. For \(L_1\), the first-moment inequality at scale \(\sqrt n\) was already known to Mendel and Naor \cite[Section~8, item~2]{MN08}; it must be distinguished from the quadratic inequality. The latter is the subject of the preprint \cite{CWX26}.

The lifting equation below retains the norms and scale of the difference inequality. Its solvability concerns particular adjoint operators; it does not define Enflo type on the dual space. The \(\ell_\infty^N\) example in \cref{rem:enflo-comparison} gives a quantitative distinction. Moreover, the exponent conjugate to the moment exponent need not equal the exponent conjugate to the cotype parameter.

\subsection{Difference and divergence operators}

Throughout, all Banach spaces are real. For a linear operator \(A\), we write \(\ker A\) and \(\ran A\) for its kernel and range. All finite sets are nonempty and, unless otherwise specified, carry normalized counting measure. The notation \(\E_{\omega\in\Omega}\) denotes normalized averaging over a finite set \(\Omega\); the averaging domain is omitted when clear. Write \(\N=\{1,2,\ldots\}\). For \(k\in\N\), write \([k]=\{1,\ldots,k\}\). For \(n\in\N\) and an even positive integer \(m\), write
\[
\mathbb Z_m=\mathbb Z/(m\mathbb Z),\qquad
G_{m,n}=\mathbb Z_m^n,\qquad
\Sigma_n=\{-1,1\}^n.
\]
Let \(e_j\) denote the \(j\)-th coordinate vector, identified with its image in \(G_{m,n}\). All additions in \(G_{m,n}\) are taken modulo \(m\). As in \cite{MN08}, \(m\) is the side length and the coordinate antipodal increment is \((m/2)e_j\). We occasionally write \(h=m/2\) for the half-side length. Fix \(1\le p<\infty\), and let \(p'\) denote the conjugate exponent of \(p\), with \(p'=\infty\) when \(p=1\). For a Banach space \(X\) and a finite set \(\Omega\), let \(L_r(\Omega;X)\), \(1\le r\le\infty\), denote the space of mappings \(F:\Omega\to X\), equipped with
\[
\|F\|_{L_r(\Omega;X)}=\left(\E_{\omega\in\Omega}\|F(\omega)\|_X^r\right)^{1/r},
\qquad 1\le r<\infty,
\]
and with the maximum norm when \(r=\infty\). We write \(L_r(\Omega)\) in the scalar case and \(X^*\) for the Banach dual of \(X\). We use \(\theta\) for the zero mapping, including the zero function or zero operator, with domain and codomain understood from context; scalar zero is denoted by \(0\).

For a mapping \(F:G_{m,n}\to X\), define the coordinate antipodal difference operator \(D_{\longop}:L_p(G_{m,n};X)\to L_p(G_{m,n}\times[n];X)\) by
\begin{equation}\label{eq:Dlong}
(D_{\longop}F)(x,j)=F(x+\tfrac m2e_j)-F(x),
\end{equation}
and the sign-difference operator \(D_{\sgn}:L_p(G_{m,n};X)\to L_p(G_{m,n}\times\Sigma_n;X)\) by
\begin{equation}\label{eq:Dsgn}
(D_{\sgn}F)(x,\varepsilon)=F(x+\varepsilon)-F(x).
\end{equation}
The corresponding normalized \(L_p\)-norms are
\begin{align*}
\norm{D_{\longop}F}_{L_p(G_{m,n}\times[n];X)}^p
&=\frac1n\sum_{j=1}^n\E_x\norm{F(x+\tfrac m2e_j)-F(x)}_X^p,\\
\norm{D_{\sgn}F}_{L_p(G_{m,n}\times\Sigma_n;X)}^p
&=\E_{x,\varepsilon}\norm{F(x+\varepsilon)-F(x)}_X^p.
\end{align*}

With respect to the canonical finite-sum dualities, the adjoints are \(D_{\longop}^*:L_{p'}(G_{m,n}\times[n];X^*)\to L_{p'}(G_{m,n};X^*)\) and \(D_{\sgn}^*:L_{p'}(G_{m,n}\times\Sigma_n;X^*)\to L_{p'}(G_{m,n};X^*)\). For \(g:G_{m,n}\times[n]\to X^*\) and \(u:G_{m,n}\times\Sigma_n\to X^*\), they are given by
\begin{align}
D_{\longop}^*g(x)&=\frac1n\sum_{j=1}^n\bigl(g(x-\tfrac m2e_j,j)-g(x,j)\bigr),             \label{eq:divlong}\\
D_{\sgn}^*u(x)&=2^{-n}\sum_{\varepsilon\in\Sigma_n}\bigl(u(x-\varepsilon,\varepsilon)-u(x,\varepsilon)\bigr). \label{eq:divsgn}
\end{align}
These are the corresponding discrete divergence operators. For \(p\le q<\infty\), define the normalized long-difference operator by
\begin{equation}\label{eq:normalized-long}
\mathcal D_{q,\longop}=\frac{n^{1/q}}mD_{\longop}.
\end{equation}
This is the normalization of \cite[Definition~1.3]{MN08}, with the sum over coordinates written as a normalized average. A sign-edge metric-cotype estimate with constant \(\Gamma\) takes the form
\begin{equation}\label{eq:sign-cotype}
\norm{\mathcal D_{q,\longop}F}_{L_p(G_{m,n}\times[n];X)}
\le\Gamma\norm{D_{\sgn}F}_{L_p(G_{m,n}\times\Sigma_n;X)}.
\end{equation}
Fix a Banach space \(X\) and \(1\le p\le q<\infty\). We say that \(X\) satisfies the sign-edge metric-cotype inequality with exponent \(p\), cotype parameter \(q\), and sharp scaling if there exist \(K,\Gamma<\infty\) such that, for every \(n\in\N\), one can choose an integer \(m_n\) divisible by \(4\), with \(m_n\le K n^{1/q}\), for which \eqref{eq:sign-cotype} holds with constant \(\Gamma\). When \(p<q\), the underlying inequality is the sign-edge analogue of weak metric cotype \(q\) with exponent \(p\) in the sense of \cite[Definition~1.3]{MN08}; here we additionally impose the sharp scaling \(m_n=O(n^{1/q})\). The comparison with the ternary-edge formulation appears in \cite[Section~4]{EMN19}.

For an operator \(R\) between finite atomic scalar \(L_r\)-spaces, \(1\le r\le\infty\), its regular norm is
\begin{equation}\label{eq:regular-norm}
\norm{R}_{\mathrm{reg}}=\sup_{N\ge1}\norm{R\otimes I_{\ell_\infty^N}}.
\end{equation}
Here \(I_Y\) denotes the identity operator on a Banach space \(Y\), and \(R\otimes I_Y\) acts on \(Y\)-valued functions by the same scalar matrix as \(R\). This norm controls the same scalar matrix under all finite-dimensional \(\ell_\infty\)-valued amplifications. Consider scalar operators
\[
T:L_p(G_{m,n}\times\Sigma_n)\longrightarrow L_p(G_{m,n}\times[n])
\]
and
\[
S:L_{p'}(G_{m,n}\times[n])\longrightarrow L_{p'}(G_{m,n}\times\Sigma_n).
\]
Set
\begin{align}
\mathfrak R_p(n,m)
&=\inf\bigl\{\norm{T}_{\mathrm{reg}}:TD_{\sgn}=D_{\longop}\bigr\},              \label{eq:Rdef}\\
\mathfrak R_{p'}^*(n,m)
&=\inf\bigl\{\norm{S}_{\mathrm{reg}}:D_{\sgn}^*S=D_{\longop}^*\bigr\}.          \label{eq:Rstar}
\end{align}
Thus \(\mathfrak R_p(n,m)\) and \(\mathfrak R_{p'}^*(n,m)\) measure the least regular norm of a single linear operator that realizes the forward or adjoint factorization, respectively. If the relevant factorization does not exist, the corresponding infimum is \(+\infty\). For background on regular operators, see \cite{Pis94}.

\subsection{Main results}

\begin{theorem}\label{thm:regular}
Let \(1\le p<\infty\). For every \(n\in\N\) and every positive integer \(m\) divisible by \(4\),
\begin{equation}\label{eq:Rexact}
\mathfrak R_p(n,m)=\mathfrak R_{p'}^*(n,m)=\frac m2.
\end{equation}
If \(n=1\), the same identity holds for every even \(m\). If \(n\ge2\) and \(m\equiv2\pmod4\), then
\[
\mathfrak R_p(n,m)=\mathfrak R_{p'}^*(n,m)=+\infty.
\]
\end{theorem}

A selection is a map assigning one solution to each datum. It is odd if \(\Phi(-g)=-\Phi(g)\), and positively homogeneous if \(\Phi(tg)=t\Phi(g)\) for \(t\ge0\). For a map between normed spaces, we write
\[
\operatorname{Lip}(\Phi)=\sup_{g\ne g'}\frac{\|\Phi(g)-\Phi(g')\|}{\|g-g'\|}
\]
for its global Lipschitz constant, allowing the value \(+\infty\).

\begin{theorem}\label{thm:endpoint}
Let \(n,N\in\N\), let \(4\mid m\), and assume that
\[
2\sqrt n\le m\le4\sqrt n.
\]
Consider maps
\[
\Phi:L_\infty(G_{m,n}\times[n];\ell_\infty^N)
\longrightarrow L_\infty(G_{m,n}\times\Sigma_n;\ell_\infty^N).
\]
The equation
\begin{equation}\label{eq:endpoint}
D_{\sgn}^*\Phi(g)=\frac{\sqrt n}{m}D_{\longop}^*g
\end{equation}
has the following properties.
\begin{enumerate}
\item There is an odd positively homogeneous selection \(\Phi\) such that \(\|\Phi(g)\|_\infty\le\tfrac32\|g\|_\infty\) for every \(g\). For every \(\eta>0\), such a selection can be chosen continuous with the bound \(\tfrac32(1+\eta)\).
\item Among all bounded linear maps satisfying \eqref{eq:endpoint}, the least operator norm is \(\sqrt n/2\).
\item Among all globally Lipschitz maps satisfying \eqref{eq:endpoint}, the least Lipschitz constant is \(\sqrt n/2\).
\end{enumerate}
All constants are independent of \(N\). In particular, the assertions hold for scalar data, by taking \(N=1\).
\end{theorem}

The first assertion is a bound on the size of each solution, not on the difference between solutions for two data. A suitable \(m\) always exists: take twice the least even integer not smaller than \(\sqrt n\); then \(2\sqrt n\le m\le4\sqrt n\).

At the \(L_\infty\) endpoint, the scalar regular norm equals the operator norm. This gives the linear lower bound in \cref{thm:endpoint}; differentiating a Lipschitz selection gives the same lower bound for its Lipschitz constant. The nonlinear norm bound follows from a first-moment \(L_1\) estimate, proved directly in Section~\ref{sec:endpoint}. We first prove the regular norm formula in Section~\ref{sec:regular}, then give the Hahn--Banach lifting identity and selection construction in Section~\ref{sec:duality}. Section~\ref{sec:higher} derives a higher-moment application from the preprint \cite{CWX26}.

\section{Exact regular factorization norms}\label{sec:regular}

We first recall the matrix formula for regular norms on finite atomic spaces.

\begin{lemma}\label{lem:regular-matrix}
Let \(\Omega,\Omega'\) be finite sets with strictly positive measures \(\mu,\nu\), and let \(R:L_r(\Omega,\mu)\to L_r(\Omega',\nu)\) be a linear operator, where \(1\le r\le\infty\). Let \(r^{\prime}\) be its conjugate exponent, with the usual endpoint conventions; the adjoint acts between the corresponding \(L_{r^{\prime}}\)-spaces. Write
\[
(Rf)(b)=\sum_{a\in\Omega}k(b,a)f(a)\mu(a),
\qquad b\in\Omega',
\]
and define \(|R|\) by replacing \(k(b,a)\) with \(|k(b,a)|\). Then
\[
\|R\|_{\mathrm{reg}}=\sup_{N\ge1}\|R\otimes I_{\ell_\infty^N}\|=\||R|\|.
\]
Moreover, \(|R^*|=|R|^*\) and \(\|R^*\|_{\mathrm{reg}}=\|R\|_{\mathrm{reg}}\).
\end{lemma}

\begin{proof}
For every finite family \(f_1,\ldots,f_N\), the pointwise inequality
\[
\max_{1\le i\le N}|Rf_i|
\le |R|\!\left(\max_{1\le i\le N}|f_i|\right)
\]
gives \(\|R\otimes I_{\ell_\infty^N}\|\le\||R|\|\). Taking the supremum over \(N\) proves one inequality.

For the reverse inequality, fix a nonzero \(f\ge0\). Index the coordinates of \(\ell_\infty^{2^{|\Omega|}}\) by the \(2^{|\Omega|}\) sign patterns \(\sigma\in\{-1,1\}^{\Omega}\), and set \(F(a)_\sigma=\sigma(a)f(a)\). Then \(\|F(a)\|_\infty=f(a)\). For each \(b\in\Omega'\), choose \(\sigma_b\) so that \(k(b,a)\sigma_b(a)=|k(b,a)|\) for every \(a\). The corresponding coordinate satisfies
\[
\bigl((R\otimes I)F(b)\bigr)_{\sigma_b}
=\sum_{a\in\Omega}|k(b,a)|f(a)\mu(a)=(|R|f)(b).
\]
Consequently,
\[
\||R|f\|_{L_r(\Omega',\nu)}
\le\|(R\otimes I)F\|_{L_r(\Omega',\nu;\ell_\infty^{2^{|\Omega|}})}
\le\|R\|_{\mathrm{reg}}\|f\|_{L_r(\Omega,\mu)}.
\]
Positivity gives
\[
\||R|\|=\sup\bigl\{\||R|f\|_{L_r(\Omega',\nu)}:f\ge0,\ \|f\|_{L_r(\Omega,\mu)}\le1\bigr\},
\]
since \(\bigl||R|f\bigr|\le |R|\,|f|\) and \(\||f|\|_{L_r}=\|f\|_{L_r}\). Hence \(\||R|\|\le\|R\|_{\mathrm{reg}}\).

Finally, the canonical dual pairings give
\[
(R^*g)(a)=\sum_{b\in\Omega'}k(b,a)g(b)\nu(b).
\]
Thus taking the absolute value of the kernel commutes with taking the adjoint, so \(|R^*|=|R|^*\). Applying the first part to \(R^*\) and using the usual equality of an operator norm and its adjoint norm gives
\[
\|R^*\|_{\mathrm{reg}}=\||R^*|\|=\||R|^*\|=\||R|\|=\|R\|_{\mathrm{reg}}.
\]
\end{proof}

The normalized factorization has the following exact norm.

\begin{corollary}\label{cor:gap}
Let \(n\in\N\), \(1\le p\le q<\infty\), and let \(4\mid m\). Every regular linear operator \(S:L_{p'}(G_{m,n}\times[n])\to L_{p'}(G_{m,n}\times\Sigma_n)\) satisfying
\[
D_{\sgn}^*S=\mathcal D_{q,\longop}^*
\]
has
\begin{equation}\label{eq:linear-cost}
\norm{S}_{\mathrm{reg}}\ge\frac{n^{1/q}}2,
\end{equation}
and equality is attained.
\end{corollary}

\begin{proof}[Proof of \cref{thm:regular}]
Put \(h=m/2\). We first compute the forward factorization constant \(\mathfrak R_p(n,m)\).

We begin with the upper bound. Assume that \(h\) is even. To express a coordinate antipodal difference as a sum of sign differences, we construct paths of length \(h\) whose endpoints differ by \(he_j\). Averaging these paths will ensure that each step has the uniform sign distribution. Fix \(j\in[n]\). We construct a random sequence \(\varepsilon^1,\ldots,\varepsilon^h\in\Sigma_n\). Choose a uniform sign \(\sigma\), set \(\varepsilon_j^r=\sigma\) for every \(r\), and, independently of \(\sigma\) and independently for each \(k\ne j\), choose uniformly a subset \(A_k\subseteq[h]\) of cardinality \(h/2\) and put
\[
\varepsilon_k^r=
 \begin{cases}
1,&r\in A_k,\\
-1,&r\notin A_k.
 \end{cases}
\]
Then
\[
\sum_{r=1}^h\varepsilon^r=\sigma h e_j,
\]
which represents \(he_j\) in \(G_{m,n}\), since \(-h\equiv h\pmod m\). Put \(s_0=0\) and \(s_r=\sum_{a=1}^r\varepsilon^a\), and define
\begin{equation}\label{eq:pathT}
(Tu)(x,j) =\E_\omega\sum_{r=1}^h u(x+s_{r-1},\varepsilon^r).
\end{equation}

Here \(\E_\omega\) denotes expectation over the random path just constructed for the fixed coordinate \(j\). For \(u=D_{\sgn}f\), the sum telescopes:
\[
\sum_{r=1}^h \bigl(f(x+s_r)-f(x+s_{r-1})\bigr)=f(x+he_j)-f(x).
\]
Hence \(TD_{\sgn}=D_{\longop}\). We now estimate the regular norm. For a Banach space \(Y\) and \(u:G_{m,n}\times\Sigma_n\to Y\), Minkowski's inequality followed by Jensen's inequality gives
\[
\begin{aligned}
\|(T\otimes I_Y)u\|_{L_p(G_{m,n}\times[n];Y)}
&\le\sum_{r=1}^h\bigl(\E_{j,x}\|\E_\omega u(x+s_{r-1},\varepsilon^r)\|_Y^p\bigr)^{1/p}\\
&\le\sum_{r=1}^h\bigl(\E_{j,x}\E_\omega\|u(x+s_{r-1},\varepsilon^r)\|_Y^p\bigr)^{1/p}.
\end{aligned}
\]
For each fixed \(j,r\), every coordinate of \(\varepsilon^r\) is a uniform sign: for \(k\ne j\), the probability that \(r\in A_k\) is \(1/2\). The coordinates are independent by construction, so \(\varepsilon^r\) is uniform on \(\Sigma_n\). Conditioned on the path, the point \(x+s_{r-1}\) is uniform on \(G_{m,n}\). Thus each expectation in the last line equals \(\E_{x,\varepsilon}\|u(x,\varepsilon)\|_Y^p\), and
\[
\|(T\otimes I_Y)u\|_{L_p(G_{m,n}\times[n];Y)}
\le h\|u\|_{L_p(G_{m,n}\times\Sigma_n;Y)}.
\]
Taking \(Y=\ell_\infty^N\) and then the supremum over \(N\) gives \(\|T\|_{\mathrm{reg}}\le h\).

To obtain the matching lower bound, we test an arbitrary factorization on a Banach-valued mapping whose sign increments have norm one and whose coordinate antipodal increments have norm \(h\). Consider the graph on \(G_{m,n}\) whose edges join \(x\) to \(x+\varepsilon\) for \(\varepsilon\in\Sigma_n\). Let \(d_\Sigma\) be the shortest-path distance within each connected component. Since \(h\) is even, the balanced paths constructed above join \(x\) to \(x+he_j\). Every sign edge has length one, and

\begin{equation}\label{eq:worddistance}
d_\Sigma(x,x+he_j)=h.
\end{equation}

Indeed, an \(L\)-step path changes the \(j\)-th coordinate by an integer of absolute value at most \(L\), while reaching \(he_j\) modulo \(m\) requires absolute value at least \(h\). Equality is attained by the balanced \(h\)-step paths used above.

For each connected component \(C\), the map \(x\mapsto(d_\Sigma(x,a))_{a\in C}\) is an isometric embedding into \(\ell_\infty^C\): the triangle inequality bounds every coordinate difference by \(d_\Sigma(x,y)\), and the coordinate \(a=x\) attains this bound. Combine these embeddings in disjoint coordinate blocks, setting all blocks other than the one containing \(x\) equal to zero. Every sign edge stays in one component, and the balanced path above shows that \(x\) and \(x+he_j\) lie in the same component. We therefore obtain \(F:G_{m,n}\to\ell_\infty^M\), where \(M=|G_{m,n}|\), such that
\[
\norm{D_{\sgn}F}_{L_p(G_{m,n}\times\Sigma_n;\ell_\infty^M)}=1,
\qquad
\norm{D_{\longop}F}_{L_p(G_{m,n}\times[n];\ell_\infty^M)}=h.
\]
If \(TD_{\sgn}=D_{\longop}\), the factorization identity remains valid after amplification by \(I_{\ell_\infty^M}\). Therefore
\[
h=\|D_{\longop}F\|=\|(T\otimes I_{\ell_\infty^M})D_{\sgn}F\|
\le\|T\|_{\mathrm{reg}}\|D_{\sgn}F\|=\|T\|_{\mathrm{reg}}.
\]
This proves \(\mathfrak R_p(n,m)=h\).

If \(n=1\), the same construction works for every even side length \(m\), without requiring \(h=m/2\) to be even: take all \(h\) increments equal to one common uniform sign \(\sigma\). Since \(\sigma h\equiv h\pmod m\), the endpoint is again the antipodal point, and the norm estimate is unchanged. The same distance argument gives the lower bound \(h\).

If \(n\ge2\) and \(m\equiv2\pmod4\), then \(h=m/2\) is odd. Define
\[
f(x)=(-1)^{x_1+x_2}.
\]
This is well defined modulo \(m\). Every sign increment changes \(x_1+x_2\) by an even integer, so \(D_{\sgn}f=\theta\). Translation by \(he_1\) changes its parity, so \(f(x+he_1)=-f(x)\) and \(D_{\longop}f\ne\theta\). The identity \(TD_{\sgn}=D_{\longop}\) is consequently impossible.

Finally, we pass to the adjoint factorization. All scalar spaces involved are finite dimensional. Taking adjoints sends \(TD_{\sgn}=D_{\longop}\) to \(D_{\sgn}^*T^*=D_{\longop}^*\). Conversely, under the canonical identifications with the biduals, \(D_{\sgn}^*S=D_{\longop}^*\) implies \(S^*D_{\sgn}=D_{\longop}\). Thus \(T\mapsto T^*\) is a bijection between the two classes of factorizations. \cref{lem:regular-matrix} shows that this correspondence preserves the regular norm, including at \(p=1\). Hence \(\mathfrak R_{p'}^*(n,m)=\mathfrak R_p(n,m)\), including when the factorization classes are empty. The three assertions now follow from the forward case.
\end{proof}

\begin{proof}[Proof of \cref{cor:gap}]
Set \(c=n^{1/q}/m>0\). The map \(S_0\mapsto cS_0\) is a bijection from the operators satisfying \(D_{\sgn}^*S_0=D_{\longop}^*\) to those satisfying \(D_{\sgn}^*S=\mathcal D_{q,\longop}^*\), with inverse \(S\mapsto c^{-1}S\). Since \(\|cS_0\|_{\mathrm{reg}}=c\|S_0\|_{\mathrm{reg}}\), \cref{thm:regular} yields
\[
\inf\bigl\{\|S\|_{\mathrm{reg}}:D_{\sgn}^*S=\mathcal D_{q,\longop}^*\bigr\}
=c\frac m2=\frac{n^{1/q}}2.
\]
Scaling an optimal unnormalized factorization attains this value.
\end{proof}

\section{Hahn--Banach duality and selections}\label{sec:duality}

The following standard extension argument gives equality of the difference and lifting constants.

\begin{lemma}\label{lem:abstract-duality}
Let \(E,U,V\) be Banach spaces and let \(A:E\to V\), \(B:E\to U\) be bounded linear operators. Set
\[
\mathfrak C(A,B) = \inf\bigl\{K\ge0:\norm{Bx}_U\le K\norm{Ax}_V \text{ for every }x\in E\bigr\},
\]
with the infimum equal to \(+\infty\) when the set is empty, and
\[
\mathfrak L(A,B) =\sup_{u^*\ne\theta} \frac{\inf\{\norm{v^*}:A^*v^*=B^*u^*\}}{\norm{u^*}}.
\]
The inner infimum is \(+\infty\) if the constraint has no solution, and an empty supremum is zero. Then \(\mathfrak L(A,B)=\mathfrak C(A,B)\). If \(\mathfrak C(A,B)<\infty\), the inner infimum is attained for every \(u^*\in U^*\).
\end{lemma}

\begin{proof}
Assume first that \(\mathfrak C(A,B)<\infty\). Then \(\ker A\subseteq\ker B\). For each \(u^*\in U^*\), the formula
\[
\lambda_{u^*}(Ax)=\langle Bx,u^*\rangle
\]
therefore defines a linear functional on \(A(E)\), and
\[
|\lambda_{u^*}(Ax)|\le\|u^*\|\|Bx\|
\le\mathfrak C(A,B)\|u^*\|\|Ax\|.
\]
The norm-preserving Hahn--Banach theorem gives an extension \(v^*\in V^*\) with \(\|v^*\|=\|\lambda_{u^*}\|\). The extension identity is precisely \(A^*v^*=B^*u^*\), so
\[
\inf\{\|v^*\|:A^*v^*=B^*u^*\}
\le\mathfrak C(A,B)\|u^*\|.
\]
Taking the supremum over nonzero \(u^*\) gives \(\mathfrak L(A,B)\le\mathfrak C(A,B)\). Moreover, every feasible \(v^*\) restricts to \(\lambda_{u^*}\) on \(A(E)\), and hence has norm at least \(\|\lambda_{u^*}\|\). The norm-preserving extension thus attains the infimum.

Conversely, suppose that \(\mathfrak L(A,B)<\infty\), and fix \(\eta>0\). For every nonzero \(u^*\in U^*\), choose \(v^*\in V^*\) such that
\[
A^*v^*=B^*u^*,\qquad
\|v^*\|\le(\mathfrak L(A,B)+\eta)\|u^*\|.
\]
For \(u^*=\theta\), take \(v^*=\theta\). For every \(x\in E\), the adjoint identity gives
\[
|\langle Bx,u^*\rangle|=|\langle Ax,v^*\rangle|
\le(\mathfrak L(A,B)+\eta)\|Ax\|\|u^*\|.
\]
Taking the supremum over \(\|u^*\|\le1\) and then letting \(\eta\downarrow0\) yields \(\mathfrak C(A,B)\le\mathfrak L(A,B)\). Thus equality holds whenever either constant is finite; otherwise both are infinite.
\end{proof}

For \(g\in L_{p'}(G_{m,n}\times[n];X^*)\), the solution fibre is
\[
\mathcal H(g)=\{u\in L_{p'}(G_{m,n}\times\Sigma_n;X^*):D_{\sgn}^*u=D_{\longop}^*g\}.
\]
The relevant difference and lifting constants are
\begin{equation}\label{eq:Cpx}
C_{p,X}(n,m)=\sup_{D_{\sgn}F\ne\theta}
\frac{\norm{D_{\longop}F}_{L_p(G_{m,n}\times[n];X)}}
{\norm{D_{\sgn}F}_{L_p(G_{m,n}\times\Sigma_n;X)}},
\end{equation}
and
\begin{equation}\label{eq:Lpxdual}
\Lambda_{p',X^*}(n,m)=\sup_{g\ne\theta}
\frac{\inf\bigl\{\norm{u}_{L_{p'}(G_{m,n}\times\Sigma_n;X^*)}:D_{\sgn}^*u=D_{\longop}^*g\bigr\}}
{\norm{g}_{L_{p'}(G_{m,n}\times[n];X^*)}}.
\end{equation}
The constant in \eqref{eq:Cpx} is the optimal \(C\) in the difference inequality \(\|D_{\longop}F\|\le C\|D_{\sgn}F\|\). We set \(C_{p,X}(n,m)=+\infty\) when \(\ker D_{\sgn}\not\subseteq\ker D_{\longop}\), and the infimum over an empty lifting fibre is \(+\infty\). Empty suprema of nonnegative quantities are zero.

\begin{proposition}\label{thm:duality}
For every Banach space \(X\), every \(1\le p<\infty\), every \(n\in\N\), and every even \(m\),
\begin{equation}\label{eq:exact-duality}
\Lambda_{p',X^*}(n,m)=C_{p,X}(n,m)
\end{equation}
in \([0,\infty]\). Hence the optimal nonlinear lifting constant agrees exactly with the corresponding Banach-valued difference constant.
\end{proposition}

\begin{proof}
Apply \cref{lem:abstract-duality} with
\[
\begin{gathered}
E=L_p(G_{m,n};X),\qquad V=L_p(G_{m,n}\times\Sigma_n;X),\\
U=L_p(G_{m,n}\times[n];X),
\end{gathered}
\]
and \(A=D_{\sgn}\), \(B=D_{\longop}\). Since the underlying measure spaces are finite, for every such space \(\Omega\) one has
\[
 L_p(\Omega;X)^*=L_{p'}(\Omega;X^*)
\]
canonically and isometrically with respect to the normalized finite-sum pairings. Under these identifications, \(A^*=D_{\sgn}^*\) and \(B^*=D_{\longop}^*\). With the kernel and empty-set conventions in \eqref{eq:Cpx} and \eqref{eq:Lpxdual}, the constants \(\mathfrak C(A,B)\) and \(\mathfrak L(A,B)\) are exactly \(C_{p,X}(n,m)\) and \(\Lambda_{p',X^*}(n,m)\), respectively. The lemma gives their equality.
\end{proof}

The normalized difference inequality has the following equivalent lifting formulation.

\begin{corollary}\label{cor:cotype}
Let \(X\) be a Banach space, let \(n\in\N\), let \(m\) be even, and fix \(1\le p\le q<\infty\) and \(0\le\Gamma<\infty\). The estimate
\[
\norm{\mathcal D_{q,\longop}F}_{L_p(G_{m,n}\times[n];X)} \le\Gamma\norm{D_{\sgn}F}_{L_p(G_{m,n}\times\Sigma_n;X)}
\]
holds for every \(F:G_{m,n}\to X\) if and only if every \(g\in L_{p'}(G_{m,n}\times[n];X^*)\) admits a solution \(u\in L_{p'}(G_{m,n}\times\Sigma_n;X^*)\) such that
\begin{equation}\label{eq:normalized-lifting}
\begin{gathered}
D_{\sgn}^*u=\mathcal D_{q,\longop}^*g,\\
\norm{u}_{L_{p'}(G_{m,n}\times\Sigma_n;X^*)}
\le\Gamma\norm{g}_{L_{p'}(G_{m,n}\times[n];X^*)}.
\end{gathered}
\end{equation}

The optimal constants in the two assertions are equal.
\end{corollary}

\begin{proof}[Proof of \cref{cor:cotype}]
Apply \cref{lem:abstract-duality} with \(A=D_{\sgn}\) and \(B=\mathcal D_{q,\longop}\) on the spaces used in the proof of \cref{thm:duality}. If the difference inequality holds with constant \(\Gamma\), the norm-preserving Hahn--Banach extension in the proof of the lemma gives a lifting for every datum with the same constant \(\Gamma\). Conversely, the lifting estimate implies the difference inequality with that constant. Thus the optimal constants coincide.
\end{proof}

\subsection{Selections in finite dimensions}

\begin{proposition}\label{prop:sections}
Let \(1\le p<\infty\), \(n\in\N\), and let \(m\) be even. Suppose that \(X\) is finite dimensional and \(C_{p,X}(n,m)<\infty\). Then the fibres in \eqref{eq:Lpxdual} admit an odd positively homogeneous selection of minimal-norm representatives. For every \(\varepsilon>0\), they also admit a continuous odd positively homogeneous selection \(\Phi_\varepsilon\) satisfying
\[
\|\Phi_\varepsilon(g)\|_{L_{p'}(G_{m,n}\times\Sigma_n;X^*)}\le(1+\varepsilon)C_{p,X}(n,m)\|g\|_{L_{p'}(G_{m,n}\times[n];X^*)}.
\]
\end{proposition}

The continuity assertion is a finite-dimensional form of the Bartle--Graves selection theorem \cite{BG52}. We give a direct construction that also preserves oddness and positive homogeneity.

\begin{proof}[Proof of \cref{prop:sections}]
Write \(A=D_{\sgn}^*\), \(B=D_{\longop}^*\), and \(C=C_{p,X}(n,m)\). By \cref{thm:duality}, each fibre \(\mathcal H(g)=\{u:Au=Bg\}\) is nonempty and has minimum norm
\[
r(g)=\min_{u\in\mathcal H(g)}\|u\|\le C\|g\|.
\]
The minimum is attained because a bounded minimizing sequence has a convergent subsequence in the finite-dimensional solution space. Fix a Euclidean norm \(|\cdot|_E\) on this space, rescaled so that \(|u|_E\le\|u\|\).

The set \(M(g)=\{u\in\mathcal H(g):\|u\|=r(g)\}\) is nonempty, compact, and convex. Let \(\Phi_0(g)\) be the unique minimizer of \(|u|_E^2\) on \(M(g)\). For every real \(t\ne0\),
\[
\mathcal H(tg)=t\mathcal H(g),\qquad r(tg)=|t|r(g),\qquad M(tg)=tM(g).
\]
Uniqueness gives \(\Phi_0(tg)=t\Phi_0(g)\). Also, \(M(\theta)=\{\theta\}\). Thus \(\Phi_0\) is odd and positively homogeneous, and chooses minimal-norm representatives.

For \(\tau>0\), define
\[
J_\tau(u)=\|u\|^2+\tau|u|_E^2,\qquad
\Psi_\tau(g)=\operatorname*{argmin}_{u\in\mathcal H(g)}J_\tau(u).
\]
Coercivity and strict convexity give a unique minimizer. Comparison with \(\Phi_0(g)\) yields
\[
\|\Psi_\tau(g)\|^2\le J_\tau(\Psi_\tau(g))
\le J_\tau(\Phi_0(g))\le(1+\tau)C^2\|g\|^2.
\]
Since \(J_\tau(tu)=t^2J_\tau(u)\), uniqueness again gives \(\Psi_\tau(tg)=t\Psi_\tau(g)\), including \(t=0\).

To verify continuity, choose a linear right inverse \(Q\) of \(A\) on \(\ran A\); it is bounded since the spaces are finite dimensional. Let \(g_k\to g\). The norm bound makes \((\Psi_\tau(g_k))\) bounded. Every \(u\in\mathcal H(g)\) has competitors
\[
u_k=u+QB(g_k-g)\in\mathcal H(g_k),\qquad u_k\to u,
\]
because \(\ran B\subseteq\ran A\). If a subsequence of \(\Psi_\tau(g_k)\) converges to \(u_*\), then \(Au_*=Bg\), and the minimizing inequalities \(J_\tau(\Psi_\tau(g_k))\le J_\tau(u_k)\), along that subsequence, give \(J_\tau(u_*)\le J_\tau(u)\) for every \(u\in\mathcal H(g)\). Uniqueness implies \(u_*=\Psi_\tau(g)\). All convergent subsequences therefore have this limit; boundedness in finite dimensions implies \(\Psi_\tau(g_k)\to\Psi_\tau(g)\).

Set \(\tau=(1+\varepsilon)^2-1\) and \(\Phi_\varepsilon=\Psi_\tau\). The asserted bound and continuity follow.
\end{proof}

\section{The endpoint separation}\label{sec:endpoint}

The first-moment estimate follows by representing \(L_1\) distances as squared Hilbert distances. We first prove the required Hilbert-valued inequality.

\begin{lemma}\label{lem:hilbert-sign}
Let \(H\) be a real Hilbert space, let \(n\in\N\), and let \(4\mid m\). Put \(h=m/2\). Every \(F:G_{m,n}\to H\) satisfies
\begin{equation}\label{eq:hilbert-sign}
\sum_{j=1}^n\E_x\|F(x+he_j)-F(x)\|_H^2
\le (h^2+2n)\E_{x,\varepsilon}\|F(x+\varepsilon)-F(x)\|_H^2.
\end{equation}
\end{lemma}

\begin{proof}
We use Fourier expansion in the complexification \(H_{\mathbb C}\) of the finite-dimensional span of the values of \(F\). Write \(i=\sqrt{-1}\). For \(k\in G_{m,n}\), set
\[
\chi_k(x)=\exp\left(\frac{2\pi i}{m}\sum_{j=1}^nk_jx_j\right),\qquad
\widehat F(k)=\E_x[\overline{\chi_k(x)}F(x)].
\]
The complexification is Hilbertian and preserves the norms of the original vectors. Parseval's identity applies to each difference. Let \(r(k)\) be the number of odd coordinates of \(k\), and put
\[
a(k)=\prod_{j=1}^n\cos\left(\frac{2\pi k_j}{m}\right).
\]
The parity of each coordinate is well defined modulo \(m\). The two sides before insertion of the constant have the expansions
\begin{align}
\sum_{j=1}^n\E_x\|F(x+he_j)-F(x)\|_H^2
&=\sum_k4r(k)\|\widehat F(k)\|_{H_{\mathbb C}}^2,\label{eq:long-symbol}\\
\E_{x,\varepsilon}\|F(x+\varepsilon)-F(x)\|_H^2
&=\sum_k2(1-a(k))\|\widehat F(k)\|_{H_{\mathbb C}}^2.\label{eq:sign-symbol}
\end{align}
Indeed, \(\chi_k(he_j)=(-1)^{k_j}\) and \(\E_\varepsilon\chi_k(\varepsilon)=a(k)\).

We claim that
\begin{equation}\label{eq:symbol-bound}
1-a(k)\ge\frac{2r(k)}{h^2+2r(k)}.
\end{equation}
When \(r(k)=0\), the right-hand side is zero and the assertion follows from \(a(k)\le1\). When \(a(k)\le0\), it follows from \(1-a(k)\ge1\). Suppose therefore that \(r(k)>0\) and \(a(k)>0\). Since \(h\) is even, for every odd \(k_j\) the distance of \(k_j\) from the nearest multiple of \(h\) is at least one. Hence
\[
\left|\cos\left(\frac{2\pi k_j}{m}\right)\right|
\le\cos\left(\frac\pi h\right)
\le1-\frac{2}{h^2}.
\]
The last inequality follows from \(1-\cos t=2\sin^2(t/2)\) and \(\sin u\ge2u/\pi\) for \(0\le u\le\pi/2\). Consequently,
\[
0<a(k)\le\left(1-\frac2{h^2}\right)^{r(k)}
\le\frac{1}{1+2r(k)/h^2}.
\]
Here the last estimate follows from Bernoulli's inequality applied to the reciprocal, or from \((1-b)^{-r}\ge1+rb\) for \(0\le b<1\) and integers \(r\ge1\). This proves \eqref{eq:symbol-bound}. Since \(r(k)\le n\), we obtain
\[
4r(k)\le2(h^2+2n)(1-a(k)).
\]
Apply this bound term by term in \eqref{eq:long-symbol} and \eqref{eq:sign-symbol} to obtain \eqref{eq:hilbert-sign}.
\end{proof}

\begin{lemma}\label{lem:first-moment}
Let \(n,N\in\N\), let \(4\mid m\), and put \(h=m/2\). Every \(F:G_{m,n}\to\ell_1^N\) satisfies
\begin{equation}\label{eq:first-moment}
\|D_{\longop}F\|_{L_1(\ell_1^N)}
\le\frac{h^2+2n}{n}\|D_{\sgn}F\|_{L_1(\ell_1^N)}.
\end{equation}
If \(2\sqrt n\le m\le4\sqrt n\), then
\begin{equation}\label{eq:first-normalized}
\left\|\frac{\sqrt n}{m}D_{\longop}F\right\|_{L_1(\ell_1^N)}
\le\frac32\|D_{\sgn}F\|_{L_1(\ell_1^N)}.
\end{equation}
The averaging domains in these norms are those of the indicated difference operators.
\end{lemma}

\begin{proof}
For \(a\in\mathbb R\), define \(J(a)\in L_2(\mathbb R)\) by
\[
J(a)(t)=
\begin{cases}
\mathbf1_{[0,a]}(t),&a\ge0,\\
-\mathbf1_{[a,0]}(t),&a<0.
\end{cases}
\]
Here \(\mathbf1_A\) is the indicator of a measurable set \(A\), and choices at endpoints have no effect. A direct comparison of intervals gives
\[
\|J(a)-J(b)\|_2^2=|a-b|.
\]
For \(v=(v_1,\ldots,v_N)\in\ell_1^N\), set \(\mathcal J(v)=(J(v_i))_{i=1}^N\) in the Hilbert direct sum \(\bigoplus_{i=1}^NL_2(\mathbb R)\), equipped with the square-sum norm. Then
\[
\|\mathcal J(v)-\mathcal J(w)\|^2=\|v-w\|_1.
\]
Apply \cref{lem:hilbert-sign} to \(\mathcal J\circ F\), and divide by \(n\). This proves \eqref{eq:first-moment}. After multiplying by \(\sqrt n/m\), its coefficient is
\[
\frac{h^2+2n}{m\sqrt n}=\frac12\left(t+\frac2t\right),\qquad t=\frac h{\sqrt n}.
\]
For \(1\le t\le2\), the inequality \((t-1)(t-2)\le0\) gives \(t+2/t\le3\), proving \eqref{eq:first-normalized}.
\end{proof}

\begin{proof}[Proof of \cref{thm:endpoint}]
By \cref{lem:first-moment} and \cref{cor:cotype}, with \(X=\ell_1^N\), \(p=1\), and \(q=2\), every datum \(g\) admits a solution of \eqref{eq:endpoint} with norm at most \(\tfrac32\|g\|_\infty\). More precisely, let \(\Psi\) be either selection furnished by \cref{prop:sections} for the unnormalized equation, and put \(\Phi=(\sqrt n/m)\Psi\). Its pointwise norm bound is at most \(3/2\), or \(\tfrac32(1+\eta)\) for the continuous selection. Oddness and positive homogeneity are preserved by this scaling.

We next compute the least norm of a linear solution operator. For a scalar matrix \(R:L_\infty(\Omega)\to L_\infty(\Omega')\),
\[
\|R\|=\max_{b\in\Omega'}\sum_{a\in\Omega}|k(b,a)|\mu(a)=\||R|\|=\|R\|_{\mathrm{reg}}.
\]
Thus \cref{cor:gap} at \(p=1,q=2\) gives the exact scalar linear lifting norm \(\sqrt n/2\). For \(\ell_\infty^N\)-valued data, let \(\mathcal S\) be any linear solution operator. Embed a scalar datum into the first coordinate, apply \(\mathcal S\), and take the first coordinate of the result. These coordinate maps have norm one and commute with both divergence operators. The resulting scalar operator \(S\) satisfies the scalar lifting identity and \(\|S\|\le\|\mathcal S\|\). Therefore \(\|\mathcal S\|\ge\sqrt n/2\). Conversely, amplifying an optimal regular scalar lifting by \(I_{\ell_\infty^N}\) gives a linear solution operator of norm at most \(\sqrt n/2\). This proves the second assertion.

Finally, suppose that \(\Phi\) is a globally Lipschitz selection. Its domain and codomain are finite-dimensional real normed spaces. Rademacher's differentiability theorem gives a point \(g_0\) at which \(\Phi\) has a Fr\'echet derivative \(L\), with
\[
\|L\|\le\operatorname{Lip}(\Phi).
\]
Differentiating \eqref{eq:endpoint} at \(g_0\) yields
\[
D_{\sgn}^*L=\frac{\sqrt n}{m}D_{\longop}^*.
\]
The second assertion therefore gives \(\operatorname{Lip}(\Phi)\ge\sqrt n/2\). The optimal linear lifting is itself Lipschitz with constant \(\sqrt n/2\), proving the third assertion.
\end{proof}

\begin{remark}\label{rem:enflo-comparison}
The uniform lifting estimate in \cref{thm:endpoint} does not imply uniform Enflo type \(r>1\) of \(\ell_\infty^N\). To see this quantitatively, let \(N=2^d\), index its coordinates by \(\sigma\in\{-1,1\}^d\), and put
\[
v_j(\sigma)=\sigma_j,\qquad f(\varepsilon)=\sum_{j=1}^d\varepsilon_jv_j.
\]
Then \(\|v_j\|_\infty=1\), \(\|f(\varepsilon)-f(-\varepsilon)\|_\infty=2d\), and \(\|f(\varepsilon)-f(\varepsilon^{(j)})\|_\infty=2\). Hence every Enflo-type \(r\) constant is at least \(d^{1-1/r}\). The lifting bound is independent of \(d\). For fixed \(d\) the space \(\ell_\infty^{2^d}\) is finite dimensional and hence \(K\)-convex, but its \(K\)-convexity constant is not bounded independently of \(d\). Classical type--cotype duality therefore gives no dimension-independent type bound here.
\end{remark}

\begin{remark}
The ordinary scalar operator norm behaves differently at \(p'=2\). Assume \(4\mid m\) and \(n\le(m/2)^2\), and put \(A=D_{\sgn}\), \(B=(\sqrt n/m)D_{\longop}\). By \cref{lem:hilbert-sign},
\[
\|Bf\|_2^2\le\left(\frac14+\frac{2n}{m^2}\right)\|Af\|_2^2\le\frac34\|Af\|_2^2.
\]
Thus \(T_0(Af)=Bf\) defines a scalar operator of norm at most \(\sqrt3/2\) on \(\ran A\). If \(P\) is the orthogonal projection onto \(\ran A\), then \(S=(T_0P)^*\) satisfies \(A^*S=B^*\) and \(\|S\|\le\sqrt3/2\), while \(\|S\|_{\mathrm{reg}}\ge\sqrt n/2\). This calculation uses only the Hilbert estimate, not \cite{CWX26}.
\end{remark}

\section{A higher-moment application}\label{sec:higher}

The higher-moment lifting bound uses \cite[Corollary~1.2]{CWX26}, stated in \eqref{eq:CWX}. The proof of \cref{thm:endpoint} is independent of that estimate.

\begin{corollary}\label{cor:L1}
Let \(n\in\N\), \(2\le q<\infty\), \(1\le p\le q\), and let \(m\ge8\) be divisible by \(4\), with \(n\le(m/2)^q\). For every \(N\in\N\) and every \(g\in L_{p'}(G_{m,n}\times[n];\ell_\infty^N)\), there exists \(u\in L_{p'}(G_{m,n}\times\Sigma_n;\ell_\infty^N)\) such that
\begin{equation}\label{eq:L1lift}
\begin{gathered}
D_{\sgn}^*u=\frac{n^{1/q}}mD_{\longop}^*g,\\
\norm{u}_{L_{p'}(G_{m,n}\times\Sigma_n;\ell_\infty^N)}
\le11\norm{g}_{L_{p'}(G_{m,n}\times[n];\ell_\infty^N)}.
\end{gathered}
\end{equation}

The least regular norm of a scalar linear operator \(S\) satisfying \(D_{\sgn}^*S=(n^{1/q}/m)D_{\longop}^*\) is \(n^{1/q}/2\). In particular, for \(p=q=2\) this regular norm is \(\sqrt n/2\), whereas the nonlinear lifting bound is at most \(11\).
\end{corollary}

\subsection{The quoted sharp estimate}

For the \(L_1\) application, we use the following estimate from \cite[Corollary~1.2]{CWX26}. If \(2\le q<\infty\), \(1\le p\le q\), \(M\ge4\) is even, \(n\le M^q\), and \(f:(\mathbb Z/M\mathbb Z)^n\to L_1\), then

\begin{equation}\label{eq:CWX}
\begin{split}
&\left(\sum_{j=1}^n\E_x\norm{f(x+\tfrac M2e_j)-f(x)}_1^p\right)^{1/p}\\
&\qquad\le 11M n^{1/p-1/q}\left(\E_{x,\delta\in\{-1,0,1\}^n}\norm{f(x+\delta)-f(x)}_1^p\right)^{1/p}.
\end{split}
\end{equation}

\subsection{Comparison of increments}

The estimate \eqref{eq:CWX}, applied on parity cosets, involves doubled ternary increments. The following coupling bounds them in terms of sign increments.

\begin{lemma}\label{lem:coupling}
For every Banach space \(X\), every \(1\le p<\infty\), and every \(F:G_{m,n}\to X\),
\[
\left( \E_{x,\delta\in\{-1,0,1\}^n} \norm{F(x+2\delta)-F(x)}_X^p \right)^{1/p} \le 2\norm{D_{\sgn}F}_{L_p(G_{m,n}\times\Sigma_n;X)}.
\]
\end{lemma}

\begin{proof}
For each \(i\in[n]\), choose a pair \((\varepsilon_i,\varepsilon_i')\in\{-1,1\}^2\) with distribution
\[
\begin{array}{c|cccc}
(\varepsilon_i,\varepsilon_i') &(1,1)&(-1,-1)&(1,-1)&(-1,1)\\
\hline
\text{probability}&1/3&1/3&1/6&1/6
\end{array}
\]
independently across coordinates. Set \(\varepsilon=(\varepsilon_i)_{i=1}^n\), \(\varepsilon'=(\varepsilon_i')_{i=1}^n\), and \(\delta=(\varepsilon+\varepsilon')/2\). Each \(\delta_i\) is uniform on \(\{-1,0,1\}\), so \(\delta\) is uniform on \(\{-1,0,1\}^n\). Both sign vectors are uniform on \(\Sigma_n\).

Let \(x\) be uniform on \(G_{m,n}\) and independent of \((\varepsilon,\varepsilon')\). The identity
\[
F(x+2\delta)-F(x)=F(x+\varepsilon+\varepsilon')-F(x+\varepsilon)+F(x+\varepsilon)-F(x)
\]
and Minkowski's inequality give
\[
\begin{aligned}
\bigl(\E\|F(x+2\delta)-F(x)\|_X^p\bigr)^{1/p}
&\le\bigl(\E\|F(x+\varepsilon+\varepsilon')-F(x+\varepsilon)\|_X^p\bigr)^{1/p}\\
&\quad+\bigl(\E\|F(x+\varepsilon)-F(x)\|_X^p\bigr)^{1/p}.
\end{aligned}
\]
Conditional on \((\varepsilon,\varepsilon')\), the translate \(x+\varepsilon\) is still uniform on \(G_{m,n}\). Thus \((x+\varepsilon,\varepsilon')\), like \((x,\varepsilon)\), has the product of the uniform distributions on \(G_{m,n}\) and \(\Sigma_n\). Each term on the right therefore equals \(\|D_{\sgn}F\|_{L_p(G_{m,n}\times\Sigma_n;X)}\), proving the claim.
\end{proof}

\begin{proof}[Proof of \cref{cor:L1}]
Put \(h=m/2\). We apply \eqref{eq:CWX} on parity cosets, each identified with a torus of side length \(h\). Let \(F:G_{m,n}\to\ell_1^N\). Every \(x\in G_{m,n}\) has a unique representation
\[
x=2y+\eta,\qquad y\in(\mathbb Z/h\mathbb Z)^n,\quad \eta\in\{0,1\}^n.
\]
For fixed \(\eta\), set \(F_\eta(y)=F(2y+\eta)\). Since \(h\) is even,
\[
x+he_j=2\left(y+\frac h2e_j\right)+\eta.
\]
Thus translation by \(he_j\) preserves the parity coset and corresponds to the antipodal translation on the torus of side length \(h\). Apply \eqref{eq:CWX} with \(M=h\) to \(F_\eta\); its hypotheses hold because \(h\ge4\), \(n\le h^q\), and \(\ell_1^N\) is an \(L_1\)-space. Taking \(p\)-th powers and averaging over \(\eta\) gives
\[
\sum_{j=1}^n\E_x\|F(x+he_j)-F(x)\|_1^p
\le (11h)^p n^{1-p/q}\E_{x,\delta\in\{-1,0,1\}^n}\|F(x+2\delta)-F(x)\|_1^p.
\]

We now recover the normalized sign-increment estimate. Divide by \(n\), take \(p\)-th roots, and apply \cref{lem:coupling}. By the definition \eqref{eq:Dlong}, this yields
\[
\|D_{\longop}F\|_{L_p(G_{m,n}\times[n];\ell_1^N)}
\le 22h n^{-1/q}\|D_{\sgn}F\|_{L_p(G_{m,n}\times\Sigma_n;\ell_1^N)}.
\]
Equivalently, since \(m=2h\), \((n^{1/q}/m)C_{p,\ell_1^N}(n,m)\le11\), with no dependence on \(N\).

Since \((\ell_1^N)^*=\ell_\infty^N\) isometrically, \cref{cor:cotype} gives \eqref{eq:L1lift}. \cref{cor:gap} gives the optimal scalar regular norm \(n^{1/q}/2\) for the same normalized identity.
\end{proof}


\begin{thebibliography}{99}
\raggedright

\bibitem{BG52}
R.~G.~Bartle and L.~M.~Graves, \emph{Mappings between function spaces}, Trans. Amer. Math. Soc. \textbf{72} (1952), no.~3, 400--413. \href{https://doi.org/10.1090/S0002-9947-1952-0047910-X}{doi:10.1090/S0002-9947-1952-0047910-X}.

\bibitem{CWX26}
Q.~Cheng, Y.~Wang, and B.~Xiang, \emph{Sharp Metric Cotype Inequalities for \(L_1\) via Nonlinear Cut Smoothing}, \href{https://arxiv.org/abs/2609.08749v2}{arXiv:2609.08749v2} [math.FA], 2026.

\bibitem{EMN19}
A.~Eskenazis, M.~Mendel, and A.~Naor, \emph{Nonpositive curvature is not coarsely universal}, Invent. Math. \textbf{217} (2019), no.~3, 833--886. \href{https://doi.org/10.1007/s00222-019-00878-1}{doi:10.1007/s00222-019-00878-1}.

\bibitem{GMN11}
O.~Giladi, M.~Mendel, and A.~Naor, \emph{Improved bounds in the metric cotype inequality for Banach spaces}, J. Funct. Anal. \textbf{260} (2011), no.~1, 164--194. \href{https://doi.org/10.1016/j.jfa.2010.08.015}{doi:10.1016/j.jfa.2010.08.015}.

\bibitem{IVV20}
P.~Ivanisvili, R.~van Handel, and A.~Volberg, \emph{Rademacher type and Enflo type coincide}, Ann. of Math. (2) \textbf{192} (2020), no.~2, 665--678. \href{https://doi.org/10.4007/annals.2020.192.2.8}{doi:10.4007/annals.2020.192.2.8}.

\bibitem{Mau03}
B.~Maurey, \emph{Type, cotype and \(K\)-convexity}, in \emph{Handbook of the Geometry of Banach Spaces}, Vol.~2, W.~B.~Johnson and J.~Lindenstrauss (eds.), North-Holland, Amsterdam, 2003, 1299--1332. \href{https://doi.org/10.1016/S1874-5849(03)80037-2}{doi:10.1016/S1874-5849(03)80037-2}.

\bibitem{MN08}
M.~Mendel and A.~Naor, \emph{Metric cotype}, Ann. of Math. (2) \textbf{168} (2008), no.~1, 247--298. \href{https://doi.org/10.4007/annals.2008.168.247}{doi:10.4007/annals.2008.168.247}.

\bibitem{NS16}
A.~Naor and G.~Schechtman, \emph{Metric \(X_p\) inequalities}, Forum Math. Pi \textbf{4} (2016), e3, 81 pp. \href{https://doi.org/10.1017/fmp.2016.1}{doi:10.1017/fmp.2016.1}.

\bibitem{Pis94}
G.~Pisier, \emph{Complex interpolation and regular operators between Banach lattices}, Arch. Math. (Basel) \textbf{62} (1994), no.~3, 261--269. \href{https://doi.org/10.1007/BF01261367}{doi:10.1007/BF01261367}.

\end{thebibliography}
\end{document}